\documentclass[a4paper,11pt]{article}
\usepackage{amsfonts}
\usepackage[T1]{fontenc}
\usepackage{microtype}
\usepackage{makeidx}
\usepackage{setspace}
\usepackage{authblk}
\usepackage{graphicx}
\usepackage[all]{xy}
\usepackage{amsmath}
\usepackage{amssymb}
\usepackage{booktabs}
\usepackage{braket}
\usepackage{array}
\usepackage{tabularx}
\usepackage{epigraph}
\usepackage{multirow}
\usepackage{indentfirst}
\usepackage{cite}
\usepackage{stmaryrd}
\usepackage{faktor}
\usepackage{titling}
\usepackage{tikz}
\usepackage{dsfont}
\usepackage{url}
\usepackage{hyperref}
\usetikzlibrary{matrix,arrows,decorations.pathmorphing}
    \renewcommand{\leq}{\leqslant}
    \renewcommand{\geq}{\geqslant}
\usepackage{amsthm}
\theoremstyle{plain}
\newtheorem{thm}{Theorem}[section]
\newtheorem{dfn}[thm]{Definition}

\newtheorem{prop}[thm]{Proposition}

\newtheorem{conge}[thm]{Conjecture}

\theoremstyle{definition}
\newtheorem{ex}[thm]{Example}

\definecolor{verdescuro}{RGB}{0, 200, 0} 
\definecolor{rossoscuro}{RGB}{200, 0, 0}
\definecolor{blu}{RGB}{0, 0, 200}

\theoremstyle{remark}
\newtheorem{oss}[thm]{Remark}

\DeclareMathOperator{\N}{\mathbb{N}}

\title{\bf{Non-extendably shellable skeleta of simplices}}
\author{}
\author{Davide Bolognini\thanks{Dipartimento di Ingegneria Industriale e Scienze Matematiche, Università Politecnica delle Marche, Ancona, Italy.
\href{mailto:d.bolognini@univpm.it}{d.bolognini@univpm.it} } \ and Paolo Sentinelli\thanks{ Dipartimento di Promozione delle Scienze Umane e della Qualità della Vita, Università Telematica San Raffaele, Roma, Italy. \href{mailto:paolosentinelli@gmail.com}{paolo.sentinelli@protonmail.com}}}
\date{}
\begin{document}

\maketitle

\vspace{-3em}

    \epigraph{Io cercai la mia fortuna \\ in un certo non so che; \\ ma ho trovato poi l'intoppo...}{Giuseppe Maria Foppa, \\ \emph{Il signor Bruschino}}

\vspace{1em}

\begin{abstract}
We disprove a long-standing open conjecture due to Simon stating that all skeleta of simplices are extendably shellable. In particular, for every $d \geq 3$ we provide a pure $d$-dimensional shellable simplicial complex which is not shelling completable.
\end{abstract}

\section{Introduction}

\vspace{0.3cm}

The simple idea of building a simplicial complex one piece after another is the reason why shellability plays a prominent role in combinatorics. It is impossible to provide an account of its importance here, given the vast literature on the subject, but it suffices to recall that this notion can be used to prove that a simplicial complex is Cohen-Macaulay over every field, connected in codimension one or with the homotopy type of a wedge of spheres. 

Every pure simplicial complex of dimension $d$ on $n$ vertices is a subcomplex of the so-called {\em $d$-skeleton} $S_{d,n}$ of the $(n-1)$-simplex. The $d$-skeleton of a simplex is a shellable simplicial complex (it is the independence complex of a uniform matroid) and it appears to be a very regular object. In a paper \cite{simon1994combinatorial} published in 1994, Simon conjectured that a property stronger than shellability, called {\em extendable shellability}, holds for $S_{d,n}$. 

\newpage
\noindent {\bf Simon's Conjecture} \cite[Conjecture 4.2.1]{simon1994combinatorial}: \begin{center}
Let $n \in \mathbb{N}$ with $0 \leq d<n$. Then $S_{d,n}$ is extendably shellable. \end{center}

A simplicial complex is extendably shellable if every partial shelling can be extended to a complete shelling, so Simon's Conjecture states that a shelling of every shellable simplicial complex can be extended to a shelling of the whole $d$-skeleton in which it lives. If $d \in \{0,1,n-2,n-1\}$ the statement is trivial. For $d=2$, Bj\"orner and Eriksson proved in \cite[Corollary 2]{bjornerRank3Matroids} the more general fact that the independence complexes of all rank $3$ matroids are extendably shellable, so, in particular, $S_{2,n}$ is, because it is the independence complex of the uniform matroid of rank $3$. They conjectured that all independence complexes of matroids have this property, but in \cite{hall2004counterexamples} Hall proved that the boundary complex of an $11$-dimensional cross-polytope (a matroid of rank $12$) is not extendably shellable. In the present paper we show that the conjecture of Bj\"orner-Eriksson is also false for rank $\geq 4$ uniform matroids, see Remark \ref{bjoerik}. Bigdeli, Yazdan Pour and Zaare-Nahandi in \cite[Corollary 3.8]{bigdeli2019decomposable} and Dochtermann in \cite[Corollary 4.4]{altrodochtermann} proved the case $d=n-3$. For a simpler proof see \cite{culbertson2020extendable}.

In \cite{bigdeli2019decomposable}, the authors proposed a conjecture implying Simon's Conjecture. In \cite{benedettibolognini} Benedetti and the first author found an infinite class of counterexamples to this stronger conjecture. 

A related line of research has focused on the so-called {\em shelling completable} simplicial complexes, i.e. shellable complexes  from which it is possible to complete the shelling of the skeleton. In \cite{DochtermannVD} it is proved that vertex-decomposable simplicial complexes are shelling completable. This result was generalized to $1$-decomposable complexes in \cite{ghosal2025extendability}, after possibly adding new vertices.

We recall here Stanley's words from \cite[p. 84]{stanley1996combinatorics}: "Though this conjecture is probably false, so far it has resisted all attempts to find a counterexample". In this paper we find such a counterexample.

In Section \ref{prelim} we provide a brief and essential background on notation and preliminaries. Section \ref{austeresec} introduces the crucial construction (see Definition \ref{echo}, Theorems \ref{austere} and \ref{shell}): given a $2$-dimensional simplicial complex $\Gamma$ with suitable properties, we construct a $3$-dimensional simplicial complex $\Delta_{\Gamma}$, which is a counterexample to Simon's Conjecture, i.e. a shellable simplicial complex that is not shelling completable. In Theorem \ref{tutteledim} we provide a contractible $d$-dimensional counterexample for $d \geq 3$. In Section \ref{sezioneshelling} we explicitly list the facets of a $3$-dimensional counterexample on $16$ vertices, providing its shelling. 

\section{Notation and preliminaries}\label{prelim}

Let $n \in \mathbb{N}:=\{1,2,\ldots\}$, $[n]:=\{1,\ldots,n\}$ and $0 \leq d < n$. A pure simplicial complex $\Delta$ on $n$ vertices of dimension $d$ is defined by a collection of subsets of $[n]$ of cardinality $d+1$ called {\em facets}.

Denote by $\mathcal{F}(\Delta)$ the set of facets of $\Delta$. Then $$\Delta=\bigcup\limits_{F\in \mathcal{F}(\Delta)}\mathcal{P}(F),$$ where $\mathcal{P}(X)$ is the power set of a set $X$. For every $S \subseteq \mathcal{F}(\Delta)$, let $\Gamma_S$ be the simplicial complex defined by $\mathcal{F}(\Gamma_S)=S$. We denote by $\mathcal{N}(\Delta)$ the set of the minimal non-faces of $\Delta$ with respect to inclusion, i.e. $$\mathcal{N}(\Delta)= \min(\mathcal{P}([n])\setminus \Delta).$$ We write a facet by listing its vertices within brackets $\{\ldots\}$ or $(\ldots)$. Given a face $F \in \Delta$, the {\em link} of $F$ in $\Delta$ is defined by $$\mathrm{link}_{\Delta}(F):=\{G \in \Delta: F \cup G \in \Delta, F \cap G=\varnothing\}.$$
The {\em $f$-vector} of $\Delta$ is $f(\Delta)=(1,f_0,f_1,\ldots,f_d)$ where $f_i:=|\{F \in \Delta: |F|=i+1\}|.$ The {\em $h$-vector} $h(\Delta)=(h_0,h_1,\ldots,h_{d+1})$ of $\Delta$ is defined by the relation $$h_i=\sum_{j=0}^i (-1)^{i-j}\binom{d+1-j}{i-j}f_{j-1},$$ for every $0 \leq i \leq d+1$.

 We denote by $S_{d,n}$ the $d$-skeleton of the $(n-1)$-simplex, i.e. the pure simplicial complex whose facets are $$\mathcal{F}(S_{d,n})=\{F\subseteq [n]: |F|=d+1\}.$$ 
 
\noindent The {\em complement complex} $\overline{\Delta}$ of a $d$-dimensional simplicial complex $\Delta \subseteq S_{d,n}$ is defined by setting $$\mathcal{F}(\overline{\Delta}) = \mathcal{F}(S_{d,n})\setminus \mathcal{F}(\Delta).$$

Recall that a pure $d$-dimensional simplicial complex is {\em shellable} if and only if there exists a {\em shelling} $F_1,\cdots,F_r$ of its facets, i.e. a linear order such that for every $j<i$, there exists $v \in F_i \setminus F_j$ and $k<i$ such that $F_i \setminus F_k=\{v\}$. The addition of the facet $F_i$ to $\Gamma_{i-1}:=\Gamma_{\{F_1,\ldots,F_{i-1}\}}$ is called $i$-th shelling step. It is easy to prove that $F_i$ contains a unique minimal non-face of $\Gamma_{i-1}$ and that the intersection $\Gamma_{i-1} \cap \Gamma_{F_i}$ is pure $(d-1)$-dimensional. 

For a shellable simplicial complex $\Delta$ we have $h_i \geq 0$ and $h_{d+1}$ equals the number of facets glued along their entire boundary in their shelling step. Moreover, $\Delta$ is contractible if and only if $h_{d+1}=0$.

Notice that if $\Delta$ is shellable, then $\mathrm{link}_{\Delta}(F)$ is shellable for every face $F \in \Delta$. In fact, given a shelling $F_1,\ldots,F_r$ of $\Delta$, a shelling of the link is $F_1 \setminus F,\ldots,F_r \setminus F$, where we keep only $F_i$'s with $F \subseteq F_i$. See \cite[Lemma 8.7]{zieglerlibro}. 

We say that $\Delta$ is {\em extendably shellable} if it is shellable and, for every $S \subseteq \mathcal{F}(\Delta)$ such that $\Gamma_S$ is shellable, there exists a shelling of $\Delta$ starting with the facets of $S$. 

The first example of a shellable $2$-dimensional simplicial complex that is not extendably shellable is given in \cite[p. 386]{simon1994combinatorial}. Other examples are given in \cite[Exercise 7.37]{ancorabjorner}, \cite[Proposition 5.7]{hachimori} and \cite[Theorem A]{moryama}. It follows from results in \cite{danarajklee} that any 2-dimensional sphere is extendably shellable. Ziegler \cite{ziegler1998shelling} showed that there exist simplicial 4-polytopes that are not extendably shellable. Moreover, in \cite{culbertson2020extendable} it is proved that all shellable $d$-dimensional complexes on at most $d+3$ vertices are extendably shellable. See also \cite[Proposition 4.5]{altrodochtermann} for a numerical condition ensuring extendable shellability. 

The following conjecture concerning the skeleta of the simplices is \cite[Conjecture 4.2.1]{simon1994combinatorial}. See also \cite[Ch. 2]{stanley1996combinatorics} and \cite[Exercise 8.24(iii)*]{zieglerlibro}.
\begin{conge} [Simon's Conjecture]
Let $n \in \mathbb{N}$ and $0 \leq d < n$. Then $S_{d,n}$ is extendably shellable.
\end{conge}

We say that a shellable pure simplicial complex $\Delta \subseteq S_{d,n}$ is {\em shelling completable} (see \cite[Definition 1.2]{DochtermannVD}) if every shelling of $\Delta$ can be completed to a shelling of $S_{d,n}$. In other words, Simon's Conjecture states that every shellable simplicial complex is shelling completable. Simon's Conjecture is known for $d \leq 2$ and $d \geq n-3$. Hence the first open case of the conjecture is for $n=7$ and $d=3$.

\section{Austere simplicial complexes}\label{austeresec}

In this section we introduce the classes of \emph{austere} and \emph{quiet} simplicial complexes, and a construction, the \emph{echo}, which provides an austere complex given a quiet one. 

\begin{dfn}
    A pure $d$-dimensional simplicial complex $\Delta \subsetneq S_{d,n}$ is {\em austere} if every $F \in \mathcal{F}(\overline{\Delta})$ contains at least two minimal non-faces of $\Delta$.
\end{dfn}

The interest of this definition is due to the fact that an austere simplicial complex is not shelling completable, because no new facet may be added to a shelling of $\Delta$ to provide a longer shelling. This provides the following formulation of Simon's Conjecture. 

\begin{conge}[Simon's Conjecture rephrased]
    Every austere simplicial complex is not shellable.
\end{conge}

There are several examples of austere simplicial complexes. The trivial ones are pairs of disjoint facets, for instance $(1,2,3)$ and $(4,5,6)$. It is not difficult to prove that every austere $2$-dimensional simplicial complex is disconnected. This does not hold for higher dimensions, as the following example shows.

\begin{ex}
Let $\Delta \subseteq S_{3,6}$ be the simplicial complex with facets $$(1,2,3,4),(1,2,5,6),(3,4,5,6).$$ This is an example of an austere connected simplicial complex. Notice that this example is not connected in codimension one.
\end{ex}

The next examples provide austere simplicial complexes with various properties.

\begin{ex}
Let $\Delta\subseteq S_{3,7}$ be the simplicial complex with facets
\begin{eqnarray*}
  (1,2,3,4),(2,3,4,5),(3,4,5,6),(4,5,6,7),(1,2,3,7),(1,2,6,7),(1,5,6,7).
\end{eqnarray*}
This is an example of an austere simplicial complex that is connected in codimension one, but it is not Cohen-Macaulay.
\end{ex}

\begin{ex}
Let $\Delta\subseteq S_{3,10}$ be the simplicial complex with facets
    \begin{eqnarray*}
  &&(1,2,3,6),
(1,2,3,8),
(1,2,3,10),
(1,2,5,6),
(1,2,6,10),
(1,2,8,10),\\
&&(1,3,6,7),
(1,3,6,10),
(1,3,8,10),
(1,4,5,9),
(1,4,8,10),
(1,4,9,10),\\
&&(1,5,6,7),
(1,5,6,9),
(1,6,9,10),
(2,3,4,6),
(2,3,4,9),
(2,3,6,10),\\
&&(2,3,8,9),
(2,3,8,10),
(2,7,8,10),
(3,4,5,9),
(3,4,6,7),
(3,4,7,9),\\
&&(3,5,8,9),
(3,5,8,10),
(4,6,7,8),
(4,7,8,10),
(4,7,9,10),
(5,6,7,8),\\
&&(5,6,8,9),
(5,7,8,10) 
\end{eqnarray*}
This is an example of an austere simplicial complex that is Cohen-Macaulay over a field of characteristic $\neq 2$. Since $\widetilde{H}_1(\Delta,\mathbb{F}_2) \neq 0$, it is not shellable.
\end{ex}

In the following proposition we prove that a $d$-dimensional austere simplicial complex has no minimal non-faces of cardinality $d+1$ and it contains all possible vertices.

\begin{prop}\label{prop:minimal1}
    Let $\Delta$ be a pure $d$-dimensional austere simplicial complex. Then for every $N \in \mathcal{N}(\Delta)$ we have $2 \leq |N| \leq d+2$ with $|N| \neq d+1$.
\end{prop}

\begin{proof}
Let $N \in \mathcal{N}(\Delta)$. It is clear that $|N| \leq d+2$, otherwise $\Delta$ has dimension greater than $d$. By contradiction, assume that $|N|=d+1$. This implies that $N \setminus \{v\} \in \Delta$, for every $v \in N$. Hence $N$ is a facet of $\overline{\Delta}$ and can be glued along its entire boundary, i.e. it contains itself as unique minimal non-face. This contradicts austerity. Assume now $|N|=1$ and let $v$ be a missing vertex of $\Delta$. Given any $(d-1)$-dimensional face $e$ of $\Delta$, then $e \cup \{v\}$ contains only $v$ as minimal non-face of $\Delta$. This contradicts the austerity of $\Delta$ and we conclude.
\end{proof}

The main goal of the next sections is to construct an austere shellable simplicial complex.

\subsection{Quiet simplicial complexes and their echoes}

Our aim is to construct austere $3$-dimensional simplicial complexes starting from pure $2$-dimensional simplicial complexes. To do this, we introduce the class of quiet simplicial complexes.

\begin{dfn}
    A pure $2$-dimensional simplicial complex $\Gamma \subsetneq S_{2,n}$  is {\em quiet} if $S_{1,n}\subseteq \Gamma$ and every $F \in \mathcal{F}(S_{3,n})$ contains at most two facets of $\Gamma$.
\end{dfn}

In the next result we provide a simple characterization of quiet simplicial complexes. 

\begin{prop}\label{formulazione link}
    Let $\Gamma$ be a pure $2$-dimensional simplicial complex with full $1$-skeleton. Then $\Gamma$ is quiet if and only if $\mathrm{link}_{\Gamma}(v)$ is a triangle-free graph for every vertex $v \in \Gamma$.
\end{prop}

\begin{proof}
To improve readability, in this proof we denote faces without brackets.

    Assume that $\Gamma$ is quiet. Let $v \in \Gamma$ be a vertex and a triangle $ab-bc-ac$ in $\mathrm{link}_{\Gamma}(v)$. Then the sets $abv,bcv,acv$ are facets of $\Gamma$, but in this case $abcv$ contains more than two facets of $\Gamma$, a contradiction.

    For the other implication, consider a quadruple $abcv$. If the induced subgraph of link on abc has at most one edge there is nothing to prove. By our assumptions, we may assume that it contains exactly two edges. Up to relabeling, assume that $ab$ and $bc$ are both facets of $\mathrm{link}_{\Gamma}(v)$ and $ac$ is not. Notice that $ac \notin \mathrm{link}_{\Gamma}(b)$, otherwise $av-cv-ac$ would be a triangle in $\mathrm{link}_{\Gamma}(b)$. Then $abcv$ contains exactly two facets of $\Gamma$ and we conclude. 
\end{proof}

Now we introduce a construction that, given $k\in \N$ and a pure $2$-dimensional simplicial complex on $[k]$, provides a pure $3$-dimensional simplicial complex on $[2k]$.  For $i\in [k]$ define $$C_i := \{2i-1,2i\}.$$

\begin{dfn}\label{echo}
    Let $\Gamma \subseteq S_{2,k}$ be a pure $2$-dimensional simplicial complex. The {\em echo} of $\Gamma$ is the pure $3$-dimensional simplicial complex $\Delta_{\Gamma} \subseteq S_{3,2k}$ defined by the following property:
$$F\in \mathcal{F}(\Delta_\Gamma) \iff \{i\in [k]: C_i\cap F \neq \varnothing\}\in \Gamma.$$
\end{dfn}

\begin{ex}
Let $k=4$ and consider the simplicial complex $\Gamma$ such that $\mathcal{F}(\Gamma)=\{(1,2,3),(2,3,4)\}$. Then  $\Delta_{\Gamma}$ is a simplicial complex on $[8]$. Its facets are

\begin{enumerate}
    \item the facets corresponding to the edges of $\Gamma$: $$C_1 \cup C_2,C_1 \cup C_3,C_2 \cup C_3,C_2 \cup C_4,C_3 \cup C_4,$$ i.e. $(1,2,3,4),(1,2,5,6),(3,4,5,6),(3,4,7,8),(5,6,7,8)$;

    \item the facets corresponding to $(1,2,3)$: 
    \begin{eqnarray*}
(1,2,3,5),(1,2,3,6),(1,2,4,5),(1,2,4,6),(1,3,4,5),(1,3,4,6),\\
  (2,3,4,5),(2,3,4,6),(1,3,5,6),(1,4,5,6),(2,3,5,6),(2,4,5,6);
  \end{eqnarray*}

    \item the facets corresponding to $(2,3,4)$: 
    \begin{eqnarray*}
(3,4,5,7),(3,4,5,8),(3,4,6,7),(3,4,6,8),(3,5,6,7),(3,5,6,8),\\
  (4,5,6,7),(4,5,6,8),(3,5,7,8),(3,6,7,8),(4,5,7,8),(4,6,7,8).
  \end{eqnarray*}
\end{enumerate}
\end{ex}

\vspace{0.2cm}

It is clear from the definition that $S_{1,k}\subseteq \Gamma \implies S_{1,2k}\subseteq \Delta_\Gamma$, 
i.e. if the $1$-skeleton of $\Gamma$ is a complete graph, then the $1$-skeleton of $\Delta_{\Gamma}$ is a complete graph, because in this case $C_i \cup C_j$ is a facet of $\Delta_{\Gamma}$ for every $i,j \in [k]$, $i \neq j$. 

\vspace{0.2cm}

Our interest in constructing echoes is due to the following result.

\begin{thm}\label{austere}
   Let $\Gamma$ be a quiet simplicial complex. Then $\Delta_{\Gamma}$ is austere. 
\end{thm}

\begin{proof}
    Consider a facet $F \in \overline{\Delta_\Gamma}$, i.e. $F \notin \Delta_\Gamma$ and $|F|=4$. From our assumptions, it follows that $\Delta_\Gamma$ has full $1$-skeleton. Then, by definition of $\Delta_\Gamma$, we have two possibilities: 
    \begin{itemize}
        \item $|\{i\in [k]: C_i \cap F \neq \varnothing\}|=3$. In this case there exists $i \in [k]$ such that $F=C_i \cup \{a,b\}$, where $a \in C_j,b \in C_k$, with $j \neq k$ different from $i$. Since $S_{1,k}\subseteq \Gamma$, we have that $C_i \cup \{a\}$ and $C_i \cup \{b\}$ are faces of $\Delta_{\Gamma}$, and $\{2i-1,a,b\}$, $\{2i,a,b\}$ are distinct minimal non-faces of $\Delta_\Gamma$ contained in $F$.
        \item $|\{i\in [k]: C_i \cap F \neq \varnothing\}|=|\{a,b,c,d\}|=4$. In this case there are at least two facets of $\overline{\Gamma}$ in $\{a,b,c,d\}$, say $\{a,b,c\}$ and $\{a,b,d\}$, because $\Gamma$ is quiet.  Then in $F$ there are at least two minimal non-faces of $\Delta_\Gamma$.
    \end{itemize}
This proves that $\Delta_\Gamma$ is austere.
\end{proof}

 Consider now a pure shellable contractible simplicial complex $\Gamma$ of dimension $2$. Given a shelling order $T_1,\ldots,T_r$ for $\Gamma$, 
 we provide a linear order of some facets of $\Delta_\Gamma$ associated to $T \in \mathcal{F}(\Gamma)$.  Since $\Gamma$ is pure and contractible, we have only three possibilities.
 
 \begin{itemize}
     \item $T=T_1=\{a,b,c\}$. Assume $a<b<c$. Order lexicographically the set $\{C_a \cup \{i,j\}: i\in C_b, j\in C_c\}$. Then continue with the lexicographic order of  $\{C_b \cup \{i,j\}: i\in C_a, j\in C_c\}$ and then with the lexicographic order of  $\{C_c \cup \{i,j\}: i\in C_a, j\in C_b\}$. End the linear order with $C_a\cup C_b < C_a \cup C_c  < C_b \cup C_c$.
     
     \item $T=\{a,b,c\}$ introduces exactly two new edges $\{a,c\}$ and $\{b,c\}$. Assume $a<b$ and  order lexicographically the set $\{C_a \cup \{i,j\}: i\in C_b, j\in C_c\}$. Then continue with the lexicographic order of  $\{C_b \cup \{i,j\}: i\in C_a, j\in C_c\}$ and then with the lexicographic order of  $\{C_c \cup \{i,j\}: i\in C_a, j\in C_b\}$. End the linear order with $C_a \cup C_c  < C_b \cup C_c$.
     
     \item $T=\{a,b,c\}$ introduces exactly one new edge $\{a,c\}$. 
     Assume $a<c$ and  order lexicographically the set $\{C_b \cup \{i,j\}: i\in C_a, j\in C_c\}$. Then continue with the lexicographic order of  $\{C_a \cup \{i,j\}: i\in C_b, j\in C_c\}$ and then with the lexicographic order of  $\{C_c \cup \{i,j\}: i\in C_a, j\in C_b\}$. End the linear order with $C_a \cup C_c$.

\end{itemize}

In the following theorem we use the linear order defined above, arising from a shelling order of a pure $2$-dimensional contractible simplicial complex $\Gamma$, to provide a shelling order for $\Delta_\Gamma$.

\begin{thm}\label{shell}
    Let $\Gamma$ be a pure contractible shellable $2$-dimensional simplicial complex. Then $\Delta_{\Gamma}$ is shellable.
\end{thm}

\begin{proof}
    Since $\Gamma$ is shellable, consider a shelling order $T_1,\ldots,T_r$. Combining the linear orders arising from any facet $T_i$ as described above, we obtain a linear order on the facets of $\Delta_\Gamma$. We prove that this is in fact a shelling order. 
    
    Consider two facets $F<G$ of $\Delta_{\Gamma}$. If $|F \cap G|=3$ there is nothing to prove. Then assume $|F \cap G| \leq 2$. 
    Let $T=\{i,j,r\}$ be a facet of $\Gamma$ from which $G$ arises and $C_i=\{i_1,i_2\}$, $C_j=\{j_1,j_2\}$, $C_r=\{r_1,r_2\}$. 
    We have three cases to treat.
\begin{enumerate}
    \item $G=C_i \cup C_j$. Consider a vertex $v \in G \setminus F$. Then $(G \setminus \{v\}) \cup \{2r\}$ is a facet of $\Delta_{\Gamma}$ and it appears before $G$ in the linear order.
    \item $G=C_i \cup \{j_1,r_1\}$ and $F$ arises from $T$. If both $F$ and $G$ contain $C_i$, we have $F \cap G=C_i$. Without loss of generality assume $j<r$. Notice that $j_1 \in G \setminus F$ because $|F\cap G|\leq 2$ and $j_2<j_1$, because $F$ and $G$ arise from the same facet $T$. Then $(G \setminus \{j_1\}) \cup \{j_2\}$ is a facet of $\Delta_{\Gamma}$ and it appears before $G$ in the linear order. On the other hand, suppose that $C_j \subseteq F$. Since $F<G$ in our linear order and $F$ arises from $T$ as the facet $G$ does, $F$ is not of the type $C_x\cup C_y$ and then there exists $i_1 \in C_i$ such that $C_j \cup \{i_1\} \subseteq F$ and $i_2 \in G \setminus F$. Therefore $(G \setminus \{i_2\}) \cup \{j_2\}=C_j \cup \{i_1,r_1\}\in \Delta_{\Gamma}$. Since the facet $F$ contains $C_j$ and $F<G$, we have that $C_j \cup \{i_1,r_1\} <G$ in our linear order. If $C_r \subseteq F$ the argument is the same.
    
    \item $G=C_i \cup \{j_1,r_1\}$ and $F$ does not arise from $T$. Then $F$ arises from a facet $H$ appearing before $T$ in the shelling order of $\Gamma$. Hence there exists $a \in T \setminus H$ and a facet $S<T$ of $\Gamma$ such that $T \setminus S=\{a\}$. There are two cases to be considered.
    \begin{itemize}
        \item $a\neq i$. If $a=j$ then $T \cap S=\{i,r\}$ and  $j_1 \in G \setminus F$. Let $S \setminus T=\{s\}$. Hence $(G \setminus \{j_1\}) \cup \{2s\}$ is a facet of  $\Delta_{\Gamma}$ and this facet appears before $G$ in our linear order. We proceed analogously if $a=r$. 
        \item $a=i$. In this case $T \cap S=\{j,r\}$. Let $S \setminus T=\{s\}$. 
        Notice that $2i \in G \setminus F$. 
        Assume first that  $\{j,r\}$ is the unique old edge. Then $\{i,j\}$ and $\{i,r\}$ are new edges introduced by $T$. In this case $(G \setminus \{2i\}) \cup \{j_2\}=C_j \cup \{2i-1,r_1\} \in \Delta_{\Gamma}$ and $C_j \cup \{2i-1,r_1\}$ 
        appears before $G$ in our linear order.

        Now assume that  $\{j,r\}$ is not unique as an old edge.
        Let $\{i,j\}$ be the other one. Then $\{i,r\}$ is the unique new edge introduced by $T$. 
        Hence $(G \setminus \{2i\}) \cup \{j_2\}=C_j \cup \{2i-1,r_1\} \in \Delta_{\Gamma}$ and $C_j \cup \{2i-1,r_1\}$  appears before $G$ in our linear order. 
        If the other old edge is $\{i,r\}$, then $\{i,j\}$ is the unique new edge introduced by $T$. Consider $(G \setminus \{2i\}) \cup \{r_2\}=C_r \cup \{2i-1,j_1\} \in \Delta_{\Gamma}$: this facet appears before $G$ in our linear order.
    \end{itemize} 
\end{enumerate}
This ensures the shellability of $\Delta_{\Gamma}$ and we are done.  
\end{proof}

In the following example, we show that in the previous result the assumption "contractible" cannot be dropped.

\begin{ex}
Let $\Gamma$ be the simplicial complex with facets $$(1,2,3),(1,2,4),(1,3,4),(2,3,4).$$ It is clearly a shellable $2$-dimensional simplicial complex but it is not contractible. Now we begin shelling the facets of $\Delta_{\Gamma}$. \begin{eqnarray*}
{(1,2,3)} &\mapsto& (1, 2, 3, 5), (1, 2, 3, 6), (1, 2, 4, 5), (1, 2, 4, 6), (1, 3, 4, 5), \\
  && (1, 3, 4, 6), (2, 3, 4, 5), (2, 3, 4, 6), (1, 3, 5, 6), (1, 4, 5, 6), \\
  && (2, 3, 5, 6), (2, 4, 5, 6), (1, 2, 3, 4), (1, 2, 5, 6), (3, 4, 5, 6) \\
{(1,2,4)} &\mapsto& (1, 2, 3, 7), (1, 2, 3, 8), (1, 2, 4, 7), (1, 2, 4, 8), (1, 3, 4, 7), \\
  && (1, 3, 4, 8), (2, 3, 4, 7), (2, 3, 4, 8), (1, 3, 7, 8), (1, 4, 7, 8), \\
  && (2, 3, 7, 8), (2, 4, 7, 8), (1, 2, 7, 8), (3, 4, 7, 8) \\
{(1,3,4)} &\mapsto& (1, 2, 5, 7), (1, 2, 5, 8), (1, 2, 6, 7), (1, 2, 6, 8), (1, 5, 6, 7), \\
  && (1, 5, 6, 8), (2, 5, 6, 7), (2, 5, 6, 8), (1, 5, 7, 8), (1, 6, 7, 8), \\
  && (2, 5, 7, 8), (2, 6, 7, 8), (5, 6, 7, 8) 
\end{eqnarray*}

The missing facets of $\Delta_{\Gamma}$ so far are
\begin{eqnarray*}
{(2,3,4)} &\mapsto& (3, 4, 5, 7), (3, 4, 5, 8), (3, 4, 6, 7), (3, 4, 6, 8), (3, 5, 6, 7), \\
  && (3, 5, 6, 8), (4, 5, 6, 7), (4, 5, 6, 8), (3, 5, 7, 8), (3, 6, 7, 8), \\
  && (4, 5, 7, 8), (4, 6, 7, 8).
\end{eqnarray*} 
It is not difficult to verify that it is not possible to add any facet corresponding to the last facet of $\Gamma$ (the situation is similar with a different shelling of $\Gamma$). Actually, the problem is structural: since $\widetilde{H}_2(\Delta_{\Gamma},\mathbb{F}_2) \neq 0$, then $\Delta_{\Gamma}$ is not shellable.
\end{ex}
 
\section{Counterexamples to Simon's Conjecture}\label{contro}
By Theorems \ref{austere} and \ref{shell}, to find a counterexample to Simon's Conjecture it suffices to find a quiet shellable contractible complex $\Gamma$. In this case the $f$-vector of $\Gamma$ is $$f(\Gamma)=\left(1,k,\binom{k}{2},f_2\right).$$ It follows that the $h$-vector of $\Gamma$ is $$h(\Gamma)=\left(1,k-3,\binom{k-2}{2},f_2-\binom{k-1}{2}\right).$$ This immediately provides the condition $f_2=\binom{k-1}{2}$, due to the fact that $\Gamma$ is contractible. From the fact that every facet of $\Gamma$ lies in exactly $k-3$ sets of cardinality $4$ and the quietness of $\Gamma$, we obtain the relation $$(k-3)\binom{k-1}{2} \leq 2\binom{k}{4},$$ i.e. $k \geq 6$. 

Assume $k=6$. Since $\Gamma$ is shellable and contractible, it has a free edge $ab$. Let $abc$ be the unique facet of $\Gamma$ containing $ab$. Let $x,y,z$ be the vertices different from $a,b,c$. In this case every $4$-subset $S \subseteq [6]$ contains exactly two facets of $\Gamma$. Consider the set $abxy$. It is clear that $abx \notin \Gamma$ and $aby \notin \Gamma$. Therefore $axy \in \Gamma$, because $S$ contains both the other two triangles. With a similar argument we obtain that $axz \in \Gamma$ and $ayz \in \Gamma$. This contradicts the quietness of $\Gamma$, because $axyz$ contains more than two facets of $\Gamma$. 

The non-existence of a quiet contractible shellable simplicial complex on $k=7$ vertices can be obtained similarly with a case by case treatment; alternatively, a SageMath computation leads to the same conclusion.  

A quiet shellable contractible $2$-dimensional simplicial complex $\Gamma$ on $8$ vertices can be found in the following way.   
Consider any Steiner system $S(3,4,8)$ (see e.g. \cite{steiner}). It has $14$ blocks, i.e. $14$ quadruples. By definition of Steiner system, every facet of the $2$-skeleton $S_{2,8}$ is contained in a unique block of the system. If we look for a quiet simplicial complex with $21$ facets, at least $7$ of the blocks contain exactly $2$ facets. One finds $14$ orbits of the automorphism group of the Steiner system acting on the set of subsets of seven blocks of the system. By choosing a representative of each orbit, picking in all possible ways two facets in each block of the representative and the other seven, we obtain all the $59$ non-isomorphic quiet simplicial complexes on $8$ vertices with $21$ facets. Then, by checking all of them, we get a unique shellable contractible quiet simplicial complex on $8$ vertices. Its $f$-vector is $(1,8,28,21)$ and up to relabeling of the vertices its facets are \begin{eqnarray*}
  &&  (1, 4, 8),
 (3, 4, 8),
 (2, 3, 4),
 (3, 4, 7),
 (3, 6, 8),
 (3, 5, 6),
 (1, 6, 8), \\
 && (3, 5, 7),
 (2, 3, 6),
 (1, 4, 7),
 (1, 6, 7),
 (1, 5, 6),
 (2, 6, 7),
 (2, 5, 7),\\
 && (2, 4, 5),
 (1, 2, 5),
 (4, 5, 8),
 (1, 2, 8),
 (2, 7, 8),
 (1, 2, 3),
 (4, 5, 6).
\end{eqnarray*}


By Theorems \ref{austere} and \ref{shell}, it follows that the echo $\Delta_{\Gamma}$ is shellable and austere. This is our first counterexample to Simon's Conjecture in dimension $3$. Its $f$-vector is $(1,16,120,280,280)$ and its $h$-vector is $(1,12,78,84,105)$. In Section \ref{sezioneshelling}, we provide an explicit shelling of this simplicial complex by following the construction of the linear order presented above.

The next result provides counterexamples to Simon's Conjecture with fewer facets than the preceding example. 

\begin{prop}\label{prop:contraibile}
Let $\Delta$ be a $d$-dimensional shellable austere simplicial complex. Then there exists a contractible, shellable simplicial complex that is not shelling completable with the same $(d-1)$-skeleton of $\Delta$.
\end{prop}

\begin{proof}
    If $\Delta$ is contractible we have nothing to show, because austere simplicial complexes are not shelling completable. If it is not contractible, then in the shelling there are $h_{d+1}(\Delta)$ facets attached along their entire boundary in the corresponding shelling step. All these facets may be removed while preserving the $(d-1)$-skeleton. For this reason, no other facets may be attached to the simplicial complex obtained, except the removed ones. By construction, the resulting complex is also shellable and this concludes the proof.
\end{proof}

In other words, given the austere shellable simplicial complex above, we may remove all the $105$ facets attached along their entire boundary (the facets in blue in the shelling of Section \ref{sezioneshelling}) to obtain a shellable contractible not shelling completable simplicial complex with $f$-vector $(1,16,120,280,175)$ and $h$-vector $(1,12,78,84,0)$.

Our last result ensures that, for every dimension $d \geq 3$, there exists a $d$-skeleton that is not extendably shellable.

\begin{thm}\label{tutteledim}
    Let $d \geq 3$. Then there exists a pure simplicial complex of dimension $d$ on $d+13$ vertices with $175$ facets that is not shelling completable. In particular, the skeleton $S_{d,d+13}$ is not extendably shellable, for $d \geq 3$.
\end{thm}

\begin{proof}
Let $\Delta$ be our $3$-dimensional shellable not shelling completable counterexample with $175$ facets. It suffices to prove that the cone $C(\Delta)$ is not shelling completable. By contradiction, assume that $C(\Delta)$ is shelling completable. Then there exists a shelling of $S_{4,17}$ starting with the facets of $C(\Delta)$. Let $v$ be the new vertex, i.e. $\mathrm{link}_{S_{4,17}}(v)=S_{3,16}$. Recall that a shelling of $\mathrm{link}_{S_{4,17}}(v)$ can be induced by a shelling of $S_{4,17}$. From the equality $\mathrm{link}_{C(\Delta)}(v)=\Delta$, it follows that $\Delta$ is shelling completable, but this is a contradiction. By iterating the cone operation we conclude.   
\end{proof}

\begin{oss}
    Notice that Theorems \ref{austere}, \ref{shell} and \ref{tutteledim} provide a method to produce shellable not vertex-decomposable simplicial complexes, thanks to \cite[Corollary 2.11]{DochtermannVD}.
\end{oss}

\begin{oss}\label{bjoerik}
    It is worth mentioning that $S_{3,16}$ is the independence complex of a uniform matroid which is not extendably shellable. This gives a negative answer to a question of \cite[Remark 1]{bjornerRank3Matroids} with a smaller example than the classical counterexample constructed by Hall \cite{hall2004counterexamples}.
\end{oss}

\begin{oss}
Let $\Delta \subseteq S_{3,16}$ be our austere counterexample with $280$ facets. It can be checked that the pure $11$-dimensional simplicial complex whose facets are the complements of the facets of $\overline{\Delta}$ gives a negative answer also to \cite[Question 10]{benedettibolognini}. Moreover, it is clear that we have also provided counterexamples to \cite[Conjecture 7.2]{doolittle2022partition} and \cite[Conjectures 3.1 and 3.2]{ficarra2025simon}. Finally, disproving Simon's Conjecture shows that in general the game of shelling on skeleta of simplices is not strongly convergent in the sense of Eriksson \cite{eriksson1996strong}. 
\end{oss}

$\,$

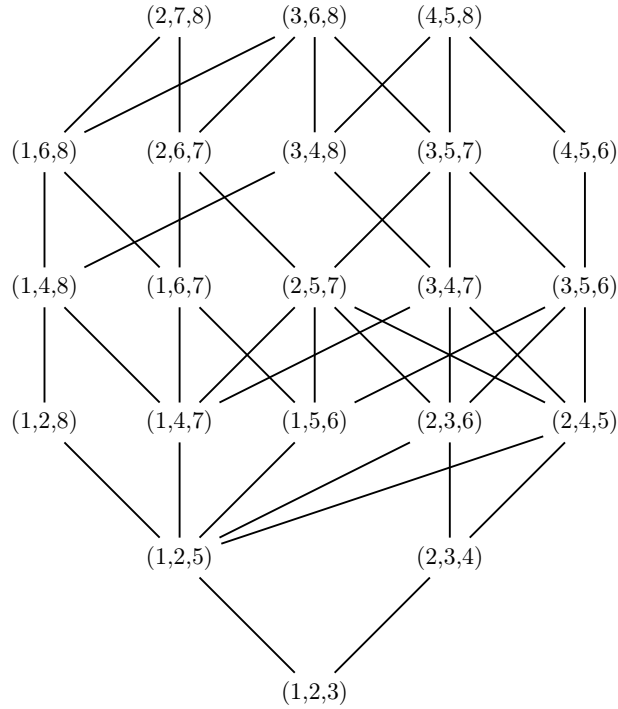
\begin{figure}[htbp] 
    \centering 
    \resizebox{0.65\textwidth}{!}{
    \begin{tikzpicture}[
        line width=0.8pt,
        posetnode/.style={
            inner sep=3pt,
            font=\small
        }
    ]

        \node[posetnode] (n123) at (0, 0) {(1,2,3)};

        \node[posetnode] (n125) at (-2, 2) {(1,2,5)};
        \node[posetnode] (n234) at (2, 2) {(2,3,4)};

        \node[posetnode] (n128) at (-4, 4) {(1,2,8)};
        \node[posetnode] (n147) at (-2, 4) {(1,4,7)};
        \node[posetnode] (n156) at (0, 4) {(1,5,6)};
        \node[posetnode] (n236) at (2, 4) {(2,3,6)};
        \node[posetnode] (n245) at (4, 4) {(2,4,5)};

        \node[posetnode] (n148) at (-4, 6) {(1,4,8)};
        \node[posetnode] (n167) at (-2, 6) {(1,6,7)};
        \node[posetnode] (n257) at (0, 6) {(2,5,7)};
        \node[posetnode] (n347) at (2, 6) {(3,4,7)};
        \node[posetnode] (n356) at (4, 6) {(3,5,6)};

        \node[posetnode] (n168) at (-4, 8) {(1,6,8)};
        \node[posetnode] (n267) at (-2, 8) {(2,6,7)};
        \node[posetnode] (n348) at (0, 8) {(3,4,8)};
        \node[posetnode] (n357) at (2, 8) {(3,5,7)};
        \node[posetnode] (n456) at (4, 8) {(4,5,6)};

        \node[posetnode] (n278) at (-2, 10) {(2,7,8)};
        \node[posetnode] (n368) at (0, 10) {(3,6,8)};
        \node[posetnode] (n458) at (2, 10) {(4,5,8)};

        \draw (n123) -- (n125);
        \draw (n123) -- (n234);

        \draw (n125) -- (n128);
        \draw (n125) -- (n147);
        \draw (n125) -- (n156);
        \draw (n125) -- (n236);
        \draw (n125) -- (n245);

        \draw (n234) -- (n236);
        \draw (n234) -- (n245);

        \draw (n128) -- (n148);

        \draw (n147) -- (n148);
        \draw (n147) -- (n167);
        \draw (n147) -- (n257);
        \draw (n147) -- (n347);

        \draw (n156) -- (n167);
        \draw (n156) -- (n257);
        \draw (n156) -- (n356);

        \draw (n236) -- (n257);
        \draw (n236) -- (n347);
        \draw (n236) -- (n356);

        \draw (n245) -- (n257);
        \draw (n245) -- (n347);
        \draw (n245) -- (n356);

        \draw (n148) -- (n168);
        \draw (n148) -- (n348);

        \draw (n167) -- (n168);
        \draw (n167) -- (n267);

        \draw (n257) -- (n267);
        \draw (n257) -- (n357);

        \draw (n347) -- (n348);
        \draw (n347) -- (n357);

        \draw (n356) -- (n357);
        \draw (n356) -- (n456);

        \draw (n168) -- (n278);
        \draw (n168) -- (n368);

        \draw (n267) -- (n278);
        \draw (n267) -- (n368);

        \draw (n348) -- (n368);
        \draw (n348) -- (n458);

        \draw (n357) -- (n368);
        \draw (n357) -- (n458);

        \draw (n456) -- (n458);

    \end{tikzpicture}%
    } \caption{The $21$ facets of the quiet simplicial complex on $8$ vertices, ordered componentwise}
   \label{sezionefigura}
\end{figure}

\newpage
\subsection{A shelling of the austere $3$-dimensional counterexample}\label{sezioneshelling}
The first facet is shown in red (three new edges). Facets introducing two new edges are shown in green, while facets introducing one new edge are shown in black. Blue facets are those attached along their entire boundary in their shelling step.
\begin{eqnarray*}
  \textcolor{rossoscuro}{(1,4,8)} &\mapsto& (1, 2, 7, 15), (1, 2, 7, 16), (1, 2, 8, 15), (1, 2, 8, 16), (1, 7, 8, 15), \\
  && (1, 7, 8, 16), (1, 7, 15, 16), (1, 8, 15, 16), (2, 7, 8, 15), \textcolor{blu}{(2, 7, 8, 16)}, \\
  && (2, 7, 15, 16), \textcolor{blu}{(2, 8, 15, 16)}, \textcolor{blu}{(1, 2, 7, 8)}, \textcolor{blu}{(1, 2, 15, 16)}, \textcolor{blu}{(7, 8, 15, 16)} \\
  \textcolor{verdescuro}{(3,4,8)} &\mapsto& (5, 7, 8, 15), (5, 7, 8, 16), (6, 7, 8, 15), (6, 7, 8, 16), (5, 7, 15, 16), \\
  && \textcolor{blu}{(5, 8, 15, 16)}, (6, 7, 15, 16), \textcolor{blu}{(6, 8, 15, 16)}, (5, 6, 7, 15), (5, 6, 7, 16), \\
  && (5, 6, 8, 15), \textcolor{blu}{(5, 6, 8, 16)}, \textcolor{blu}{(5, 6, 7, 8)}, \textcolor{blu}{(5, 6, 15, 16)} \\
  \textcolor{verdescuro}{(2,3,4)} &\mapsto& (3, 5, 6, 7), (3, 5, 6, 8), (4, 5, 6, 7), (4, 5, 6, 8), (3, 5, 7, 8), \\
  && \textcolor{blu}{(3, 6, 7, 8)}, (4, 5, 7, 8), \textcolor{blu}{(4, 6, 7, 8)}, (3, 4, 5, 7), (3, 4, 5, 8), \\
  && (3, 4, 6, 7), \textcolor{blu}{(3, 4, 6, 8)}, \textcolor{blu}{(3, 4, 5, 6)}, \textcolor{blu}{(3, 4, 7, 8)} \\
  \textcolor{verdescuro}{(3,4,7)} &\mapsto& (5, 6, 7, 13), (5, 6, 7, 14), (5, 6, 8, 13), (5, 6, 8, 14), (5, 7, 8, 13), \\
  && (5, 7, 8, 14), \textcolor{blu}{(6, 7, 8, 13)}, \textcolor{blu}{(6, 7, 8, 14)}, (5, 7, 13, 14), (5, 8, 13, 14), \\
  && (6, 7, 13, 14), \textcolor{blu}{(6, 8, 13, 14)}, \textcolor{blu}{(5, 6, 13, 14)}, \textcolor{blu}{(7, 8, 13, 14)} \\
  \textcolor{verdescuro}{(3,6,8)} &\mapsto& (5, 6, 11, 15), (5, 6, 11, 16), (5, 6, 12, 15), (5, 6, 12, 16), (5, 11, 15, 16), \\
  && (5, 12, 15, 16), \textcolor{blu}{(6, 11, 15, 16)}, \textcolor{blu}{(6, 12, 15, 16)}, (5, 11, 12, 15), (5, 11, 12, 16), \\
  && (6, 11, 12, 15), \textcolor{blu}{(6, 11, 12, 16)}, \textcolor{blu}{(5, 6, 11, 12)}, \textcolor{blu}{(11, 12, 15, 16)} \\
  \textcolor{verdescuro}{(3,5,6)} &\mapsto& (5, 6, 9, 11), (5, 6, 9, 12), (5, 6, 10, 11), (5, 6, 10, 12), (5, 9, 11, 12), \\
  && (5, 10, 11, 12), \textcolor{blu}{(6, 9, 11, 12)}, \textcolor{blu}{(6, 10, 11, 12)}, (5, 9, 10, 11), (5, 9, 10, 12), \\
  && (6, 9, 10, 11), \textcolor{blu}{(6, 9, 10, 12)}, \textcolor{blu}{(5, 6, 9, 10)}, \textcolor{blu}{(9, 10, 11, 12)} \\
  (1,6,8) &\mapsto& (1, 11, 15, 16), (1, 12, 15, 16), (2, 11, 15, 16), (2, 12, 15, 16), (1, 2, 11, 15), \\
  && \textcolor{blu}{(1, 2, 11, 16)}, (1, 2, 12, 15), \textcolor{blu}{(1, 2, 12, 16)}, (1, 11, 12, 15), \textcolor{blu}{(1, 11, 12, 16)}, \\
  && (2, 11, 12, 15), \textcolor{blu}{(2, 11, 12, 16)}, \textcolor{blu}{(1, 2, 11, 12)} \\
  (3,5,7) &\mapsto& (5, 6, 9, 13), (5, 6, 9, 14), (5, 6, 10, 13), (5, 6, 10, 14), (5, 9, 10, 13), \\
  && (5, 9, 10, 14), \textcolor{blu}{(6, 9, 10, 13)}, \textcolor{blu}{(6, 9, 10, 14)}, (5, 9, 13, 14), (5, 10, 13, 14), \\
  && \textcolor{blu}{(6, 9, 13, 14), (6, 10, 13, 14), (9, 10, 13, 14)} \\
  (2,3,6) &\mapsto& (3, 5, 6, 11), (3, 5, 6, 12), (4, 5, 6, 11), (4, 5, 6, 12), (3, 4, 5, 11), \\
  && (3, 4, 5, 12), \textcolor{blu}{(3, 4, 6, 11), (3, 4, 6, 12)}, (3, 5, 11, 12), \textcolor{blu}{(3, 6, 11, 12)}, \\
  && (4, 5, 11, 12), \textcolor{blu}{(4, 6, 11, 12)}, \textcolor{blu}{(3, 4, 11, 12)} \\
  (1,4,7) &\mapsto& (1, 7, 8, 13), (1, 7, 8, 14), (2, 7, 8, 13), (2, 7, 8, 14), (1, 2, 7, 13) \\
  && (1, 2, 7, 14), \textcolor{blu}{(1, 2, 8, 13), (1, 2, 8, 14)}, (1, 7, 13, 14), \textcolor{blu}{(1, 8, 13, 14)}, \\
  && (2, 7, 13, 14), \textcolor{blu}{(2, 8, 13, 14), (1, 2, 13, 14)}
\end{eqnarray*}

\begin{eqnarray*}
(1,6,7) &\mapsto& (1, 2, 11, 13), (1, 2, 11, 14), (1, 2, 12, 13), (1, 2, 12, 14), (1, 11, 12, 13), \\
  && (1, 11, 12, 14), \textcolor{blu}{(2, 11, 12, 13), (2, 11, 12, 14)}, (1, 11, 13, 14), (1, 12, 13, 14), \\
  && \textcolor{blu}{(2, 11, 13, 14), (2, 12, 13, 14), (11, 12, 13, 14)}\\
  (1,5,6) &\mapsto& (1, 9, 11, 12), (1, 10, 11, 12), (2, 9, 11, 12), (2, 10, 11, 12), (1, 2, 9, 11), \\
  && \textcolor{blu}{(1, 2, 9, 12)}, (1, 2, 10, 11), \textcolor{blu}{(1, 2, 10, 12)}, (1, 9, 10, 11), \textcolor{blu}{(1, 9, 10, 12)}, \\
  && (2, 9, 10, 11), \textcolor{blu}{(2, 9, 10, 12), (1, 2, 9, 10)} \\
  (2,6,7) &\mapsto& (3, 11, 12, 13), (3, 11, 12, 14), (4, 11, 12, 13), (4, 11, 12, 14), (3, 4, 11, 13), \\
  && (3, 4, 11, 14), \textcolor{blu}{(3, 4, 12, 13), (3, 4, 12, 14)}, (3, 11, 13, 14), \textcolor{blu}{(3, 12, 13, 14)}, \\
  && (4, 11, 13, 14), \textcolor{blu}{(4, 12, 13, 14), (3, 4, 13, 14)} \\
  (2,5,7) &\mapsto& (3, 9, 13, 14), (3, 10, 13, 14), (4, 9, 13, 14), (4, 10, 13, 14), (3, 4, 9, 13), \\
  && \textcolor{blu}{(3, 4, 9, 14)}, (3, 4, 10, 13), \textcolor{blu}{(3, 4, 10, 14)}, (3, 9, 10, 13), \textcolor{blu}{(3, 9, 10, 14)}, \\
  && (4, 9, 10, 13), \textcolor{blu}{(4, 9, 10, 14), (3, 4, 9, 10)} \\
  (2,4,5) &\mapsto& (3, 4, 7, 9), (3, 4, 7, 10), (3, 4, 8, 9), (3, 4, 8, 10), (3, 7, 8, 9), \\
  && (3, 7, 8, 10), \textcolor{blu}{(4, 7, 8, 9), (4, 7, 8, 10)}, (3, 7, 9, 10), (3, 8, 9, 10), \\
  && \textcolor{blu}{(4, 7, 9, 10), (4, 8, 9, 10), (7, 8, 9, 10)} \\
  (1,2,5) &\mapsto& (1, 3, 9, 10), (1, 4, 9, 10), (2, 3, 9, 10), (2, 4, 9, 10), (1, 2, 3, 9), \\
  && \textcolor{blu}{(1, 2, 3, 10)}, (1, 2, 4, 9), \textcolor{blu}{(1, 2, 4, 10)}, (1, 3, 4, 9), \textcolor{blu}{(1, 3, 4, 10)}, \\
  && (2, 3, 4, 9), \textcolor{blu}{(2, 3, 4, 10)}, \textcolor{blu}{(1, 2, 3, 4)} \\
  (4,5,8) &\mapsto& (7, 8, 9, 15), (7, 8, 9, 16), (7, 8, 10, 15), (7, 8, 10, 16), (7, 9, 10, 15), \\
  && (7, 9, 10, 16), \textcolor{blu}{(8, 9, 10, 15), (8, 9, 10, 16)}, (7, 9, 15, 16), (7, 10, 15, 16), \\
  && \textcolor{blu}{(8, 9, 15, 16), (8, 10, 15, 16), (9, 10, 15, 16)} \\
  (1,2,8) &\mapsto& (1, 2, 3, 15), (1, 2, 3, 16), (1, 2, 4, 15), (1, 2, 4, 16), (1, 3, 4, 15), \\
  && (1, 3, 4, 16), \textcolor{blu}{(2, 3, 4, 15), (2, 3, 4, 16)}, (1, 3, 15, 16), (1, 4, 15, 16), \\
  && \textcolor{blu}{(2, 3, 15, 16), (2, 4, 15, 16), (3, 4, 15, 16)} \\
  (2,7,8) &\mapsto& (3, 4, 13, 15), (3, 4, 13, 16), (3, 4, 14, 15), (3, 4, 14, 16), (3, 13, 14, 15), \\
  && (3, 13, 14, 16), \textcolor{blu}{(4, 13, 14, 15), (4, 13, 14, 16)}, (3, 13, 15, 16), (3, 14, 15, 16), \\
  && \textcolor{blu}{(4, 13, 15, 16), (4, 14, 15, 16), (13, 14, 15, 16)} \\
  (1,2,3) &\mapsto& (1, 3, 4, 5), (1, 3, 4, 6), (2, 3, 4, 5), (2, 3, 4, 6), (1, 2, 3, 5), \\
  && (1, 2, 3, 6), \textcolor{blu}{(1, 2, 4, 5), (1, 2, 4, 6)}, (1, 3, 5, 6), \textcolor{blu}{(1, 4, 5, 6)}, \\
  && (2, 3, 5, 6), \textcolor{blu}{(2, 4, 5, 6), (1, 2, 5, 6)} \\
  (4,5,6) &\mapsto& (7, 9, 10, 11), (7, 9, 10, 12), (8, 9, 10, 11), (8, 9, 10, 12), (7, 8, 9, 11), \\
  && (7, 8, 9, 12), \textcolor{blu}{(7, 8, 10, 11), (7, 8, 10, 12)}, (7, 9, 11, 12), \textcolor{blu}{(7, 10, 11, 12)}, \\
  && (8, 9, 11, 12), \textcolor{blu}{(8, 10, 11, 12), (7, 8, 11, 12)}
\end{eqnarray*}

\newpage 

{\bf AI disclosure}. At an early stage of this project, a candidate for a quiet shellable contractible simplicial complex was identified with the assistance of Chat GPT 5.5. The present version obtains the complex through the exhaustive computational procedure described above. All computations were performed with SageMath and Macaulay2.

\bibliographystyle{plain}
\bibliography{bibliografia}
\end{document}